\documentclass[12pt,oneside]{amsart}
\usepackage{amscd,amsmath,amssymb,amsfonts,a4}
\usepackage{hyperref}

\newlength{\rulebreite}

\def\timesover#1#2#3{\ \xymatrix@1@=0pt@M=0pt{ _{#1}&\times&_{#2} \\& ^{#3}&}\ }
\def\otimesover#1#2#3{\ \xymatrix@1@=0pt@M=0pt{ _{#1}&\otimes&_{#2} \\& ^{#3}&}\ }

\usepackage[all]{xy}
\theoremstyle{plain}
\newtheorem{thm}{Theorem}
\newtheorem{lem}[thm]{Lemma}
\newtheorem{cor}[thm]{Corollary}
\newtheorem{prop}[thm]{Proposition}
\theoremstyle{definition}
\newtheorem{defn}[thm]{Definition}

\newtheorem{rmk}[thm]{Remark}
\newtheorem{rmks}[thm]{Remarks}

\numberwithin{thm}{section}
\numberwithin{equation}{section}

\newcommand{\surj}{\twoheadrightarrow}

\newcommand{\Pic}{{\rm Pic}}

\newcommand{\Spec}{{\rm Spec \,}}

\newcommand{\sD}{{\mathcal D}}

\newcommand{\sL}{{\mathcal L}}

\newcommand{\sO}{{\mathcal O}}

\newcommand{\C}{{\mathbb C}}

\newcommand{\E}{{\mathbb E}}
\newcommand{\F}{{\mathbb F}}

\newcommand{\N}{{\mathbb N}}

\newcommand{\Q}{{\mathbb Q}}

\newcommand{\Z}{{\mathbb Z}}

\newcommand{\et}{{\acute{e}t}}

\begin{document}

\title[Stratified bundles]{Simply connected projective manifolds in characteristic $p>0$ have no nontrivial stratified bundles}
\author{H\'el\`ene Esnault}
\address{
Universit\"at Duisburg-Essen, Mathematik, 45117 Essen, Germany}
\email{esnault@uni-due.de}
\author{Vikram Mehta}
\address{Mathematics, Tata Institute, Homi Bhabha Road, Mumbai 400005, India}
\email{vikram@math.tifr.res.in}
\date{ March 26, 2010}
\thanks{The first author is supported by  the DFG Leibniz Preis, the SFB/TR45 and the ERC Advanced Grant 226257}
\begin{abstract} We show that 
simply connected projective manifolds in characteristic $p>0$ have no nontrivial stratified bundles. This gives a positive answer to a conjecture by D. Gieseker (1975). The proof uses Hrushovski's theorem on periodic points.

\end{abstract}
\maketitle
\section{Introduction}
Let $X$ be a smooth complex variety.
 The category of bundles with integrable connections on $X$
is the full subcategory of the category  of coherent $\sD_X$-modules which are $\sO_X$-coherent as well. 
It  is a $\C$-linear abelian rigid category. If $X$ is projective or if we restrict to connections which are regular singular at infinity, then it is equivalent by the Riemann-Hilbert correspondence to the $\C$-linear abelian rigid category of local systems of complex vector spaces (\cite{Del}).  Upon neutralizing those categories by the choice of a point $x\in X(\C)$, the Riemann-Hilbert correspondence translates via the Tannaka formalism into an isomorphism between the proalgebraic completion of the topological fundamental group $\pi_1^{\rm top}(X,x)$  and the Tannaka pro-algebraic group $\pi_1^{{\rm strat}}(X,x)$ of flat bundles. Malcev (\cite{Mal}) and Grothendieck (\cite{Gr}) showed  that if the \'etale fundamental group 
$\pi_1^{{\rm \et}}(X,x)$ is trivial, that is if $X$ does not have any nontrivial connected finite \'etale covering, then $\pi_1^{{\rm strat}}(X,x)$ is trivial as well, thus there are no nontrivial flat bundles. The proof has nothing to do with flat bundles, but with the fact that $\pi_1^{{\rm top}}(X,x)$ is an abstract group of finite type, and that, as a consequence of  the Riemann existence theorem,   $\pi_1^{{\rm \et}}(X,x)$ is its profinite completion. The theorem says that if the profinite completion of a group of finite type is trivial, so is its proalgebraic completion.

Let $X$ now be a smooth  variety defined over a perfect field $k$ of characteristic $p>0$. The  full subcategory of the category  of coherent $\sD_X$-modules which are $\sO_X$-coherent is again a $k$-linear rigid tensor 
category. Katz shows \cite[Theorem~1.3]{G} that it is equivalent to the category of objects
$E=(E_n, \sigma_n)_{n\in \N}$ where $E_n$ is a vector bundle,
$\sigma_n: F^*E_{n+1}\simeq E_n$ is a $\sO_X$-linear isomorphism, and where the morphisms respect all the structures. Let us call it the category 
${\sf Strat}(X)$ of stratified bundles. One neutralizes the category via the choice of a rational point $i_x: x\to X(k)$, defining the functor $\omega_x: {\sf Strat}(X)\to {\sf Vec}_k, \ \omega_x(E)=i_x^*E_{0}$. This defines  the pro-algebraic group $\pi_1^{{\rm strat}}(X,x)={\rm Aut}^{\otimes }(\omega_x)$. 

For a rational point $x\in X(k)$, we denote by 
 $\bar x$ a geometric point above it.
The purpose of this article is to show, when $X$ is projective, the analog in characteristic $p>0$ of Malcev-Grothendieck theorem  (see Theorem \ref{thm3.15}):
\begin{thm} \label{thm1.1}
 Let $X$ be a smooth connected projective variety defined over a perfect field $k$ of characteristic $p>0$. Let $\bar x\in X $ be a geometric point. If $\pi_1^{{\rm \et}}(X \otimes_k \bar k,\bar x)=\{1\}$, there are no nontrivial stratified bundles. 
\end{thm}
\noindent 
This gives a (complete) positive answer to Gieseker's conjecture \cite[p.8]{G}. If $X$ has a rational point $x\in X(k)$, one can rephrase  by saying that under the assumptions of the theorem, $\pi_1^{{\rm strat}}(X,x)=\{1\}$. The theorem implies that if $\pi_1^{{\rm \et}}(X \otimes_k \bar k,\bar x)$ is a finite group, then irreducible stratified bundles come from representations of 
$\pi_1^{{\rm \et}}(X \otimes_k \bar k,\bar x)$ (see Corollary \ref{cor3.16}).\\[.1cm] 
Dos Santos \cite{dS} studied the $k$-pro-group $\pi_1^{{\rm strat}}(X,x)$ when $k=\bar k$.   He showed that all quotients in $GL(\omega_x(E))$ are smooth algebraic groups (\cite[Corollary~12]{dS}, see also 
\cite{Mau}).
In fact, the proof is written only for the finite part, but it applies more generally. This is an important fact pleading in favor of the conjecture. \\[.2cm]
We now describe the philosophy of the proof of Theorem \ref{thm1.1}. Since there is no known  group of finite type which  controls $\pi_1^{{\rm strat}}(X,x)$ and $\pi_1^{{\rm \et}}(X,x)$, it is impossible to adapt Grothendieck's proof. Instead, one can first think of the full subcategory spanned by rank one objects. Over $\bar k$, the maximal abelian quotient $\pi_1^{{\rm ab}}(X,x)$ of $\pi_1^{{\rm \et}}(X,x)$ nearly controls the Picard variety: if $\pi^{{\rm ab}}_1(X, x)=\{1\}$,  then, for any prime $\ell \neq p$,  the $\ell$-adic Tate module $\varprojlim_{n} \Pic^0(X)(\bar k)[\ell^n]$ is trivial, thus $\Pic^\tau(X)$ is finite. So a rank one stratified bundle $L=(L_n, \sigma_n)_{n\in \N}$ on $X\otimes \bar k$ must have the property that for a strictly increasing sequence $n_i, \ i\ge 0$, the line bundles $L_{n_i}$ are all isomorphic, thus all $L_{n_i}$ are fixed by some power of the Frobenius, so define a Kummer covering of $X$, which then has to be trivial by our assumption. This implies that  all the $L_n, \ n\ge 0$ are trivial. Since on $X$ proper, a stratified bundle $E=(E_n, \sigma_n)_{n\in \N}$ is uniquely recognized by the bundles $E_n$ (\cite[Proposition~1.7]{G}), this shows the statement. \\[.1cm]
Thus the rank one case does not rely directly on the Tannaka property
of the category. It rather uses the representability of the Picard functor together with the fact that those bundles with are fixed by a power of the Frobenius  defined \'etale coverings on one hand, and are dense in $\Pic^\tau(X\otimes \bar k)$ on the other. We try to follow the same idea in the higher rank case.\\[.2cm]
We now describe the main steps  of the proof over $k=\bar k$. In order to be able to use moduli of vector bundles, we first reduce the problem to the case where all the underlying bundles  $E_n$ of the stratified bundles $E=(E_n, \sigma_n)_{n\in \N}$ are stable with vanishing Chern classes. Here  we use $\mu$- (or slope) stability with respect to a fix polarization, which is defined by Mumford by the growth of the degree of  subsheaves. 
Brenner-Kaid (\cite[Lemma~2.2]{BK}) show that if $E$ is a stratified bundle, then
$E_n$ is semistable of slope $0$ for $n$ large.
 Using Langer's boundedness \cite[Theorem~4.2]{L}, the authors apply the same argument as for the classical one in rank 1 sketched above, to conclude that if all the $E_n$ are defined over the same finite field $\F_q$ and $\pi^{{\rm \et}}_1(X_{\bar \F_q})=0$, then   all stratified bundles are trivial (\cite[Lemma~2.4]{BK}). 
We show in general that the stratified bundle  $E(n_0)=(E_{n_0+m},\sigma_{n_0+m})_{m\in \N}$, for $n_0$ large enough,  is always a successive extension of stratified bundles  $(U_m, \tau_m)_{m\in \N}$ such that all the $U_m$ are $\mu$-stable.  
This,  together with  dos Santos' theorem \cite[(9)]{dS}, describing with the projective system $\varprojlim_n H^1(X, \sO_X)$ over the Frobenius, the full subcategory of ${\sf Strat}(X)$ spanned by successive extensions of the trivial object by itself, is enough to perform the reduction (see Proposition \ref{prop2.4}). \\[.1cm]
The moduli scheme $M$ of $\chi$-stable torsionfree sheaves of Hilbert polynomial $p_{\sO_X}$ and of rank $r>0$  was constructed by Gieseker in \cite[Theorem~0.3]{G1} in dimension 2 over $k=\bar k$, and by Langer in \cite[Theorem~4.1]{L2} in general.
Here $\chi$-stability is taken with respect to a fix polarization, and is defined by Gieseker \cite[Section~0]{G1} by the growth of the Hilbert polynomial of subsheaves.  
Examples of $\chi$-stable torsionfree sheaves with  Hilbert polynomial $p_{\sO_X}$ are torsionfree sheaves with $\mu=0$ \cite[Lemma~1.2.13]{HL}.\\[.1cm]
While on $\Pic(X)$, the {\it Verschiebung} morphism, which, to a line bundle $L$,  assigns its Frobenius pullback $V(L):=F^*(L)\simeq L^p$, is well defined, on
$M$ it is not.   As is well known \cite[Theorem~1]{Gstable}, even if $E$ is stable, the bundle $F^*(E)$ need not be stable. 
We define 
the sublocus $M^s$ of  $\mu$-stable points $[E]$ for which there is a stratified bundle $(E_n, \sigma_n)_{n\in \N}$ with $E_0\cong E$.
 The bundles $E_n$ are necessarily $\mu$-stable, thus define points $[E_n]\in M^s$. In particular,  $[E]=V([E_1])$ and $M^s$ lies in the image of the sublocus of $M^s$ on which $V$ is well defined. If we assume that  there are nontrivial stratified bundles with $[E_n]\in M$,  this sublocus is not empty. We define $N\subset M$ to be the Zariski closure of $M^s$. We show that the Verschiebung is a rational dominant morphism $\xymatrix{ V: N\ar@{.>}[r] & N}$ (see Lemma \ref{lem3.8}). \\[.1cm]
 On the other hand, there is a smooth affine variety $S$, defined over $\F_p$,
 such that $X$ has a smooth model $X_S\to S$, 
the Frobenius $F_k: \Spec k\to \Spec k$ descends to the absolute Frobenius $F_S: S\to S$, and the absolute Frobenius $F: X\to X$ over $F_k$ descends to the absolute Frobenius $F_{X_S}: X_S\to X_S$ over $F_S$. 
By the representability theorem \cite[Theorem~4.1]{L2}, $M$ has a model $M_S$ with the property that for all morphisms $T\to S$ of $\F_p$-varieties, $M_S\times_S T=M_T$, where $M_T$ is the moduli of stable vector bundles of the same Hilbert polynomial and the same rank $r>1$ on $X_T\to T$, where $X_T=X_S\times_S T$.  We  take $T$ smooth so that $N\subset M$ has a model $N_T\subset M_T$, and its connected components are also defined over $T$.  We further require that $V$ has a model  
$\xymatrix{ V_T: N_T\ar@{.>}[r] & N_T}$. There is a dense open subvariety  $T^0\subset T$ such that for all closed points $t\in T^0$, the reduction $\xymatrix{ V_t: N_t\ar@{.>}[r] & N_t}$
is still defined as a rational map, is dominant, and some power fixes the irreducible components of $N_t$. 
We show that 
$V_t$ is the Verschiebung from $M_t$ restricted to $N_t$ (see Lemma \ref{lem3.10}).
 \\[.1cm]
Replacing $N_t$ by an irreducible component $Y$ say, 
we denote by $\Gamma\subset Y\times_t Y$ the Zariski closure of the graph, where it is defined,  of a power of $V_t$ respecting $Y$. 
One can  apply Hrushovski's fundamental theorem \cite[Corollary~1.2]{H} to find a dense subset of closed points of the shape $(u, \Phi_{q(t)}^m(u))\in \Gamma$, where $\Phi_{q(t)}$ is the geometric Frobenius of $N_t$ raising coordinates to the $q(t)$-th power, and where $\F_{q(t)}$ is the residue field of $t$.   From this and the representability theorem {\it loc. cit.},  one deduces that if $N$ is not empty, $N_t \otimes_{\F_q} \bar{\F}_q $ contains closed  points which are fixed under the Verschiebung (see Theorem \ref{thm3.14}). Since by Grothendieck's specialisation theorem   \cite[Th\'eor\`eme~3.8]{SGA1}, $\pi^{{\rm \et}}_1(X_t\otimes_{\F_q} \bar{\F}_q)=\{1\}$, we conclude that $N_t$ is empty, so thus is $N$. \\[.1cm]
The reduction to the stable case is written in section 2, the proof is performed  in section 3. 
In fact in section 3, we do a bit more. We show that torsion points (see Definition \ref{defn3.12}) are dense in good models of stratified schemes (see Definition \ref{defn3.4} and Theorem \ref{thm3.14}). 
In section 4, we make a few remarks and raise some questions. In particular, if $X$ is quasi-projective, in view to Grothendieck's theorem in characteristic 0, one would expect a relation between the fundamental group 
 $\pi_1^{{\rm \et}}(X\otimes_k \bar k, \bar x)$, and  tame stratified bundles.   Without resolution of singularities and without theory of canonical extension like Deligne's one \cite{Del} in characteristic 0, our method of proof can't be extended to this case. \\[.3cm]
{\it Acknowledgements:} 
The proof of Theorem \ref{thm1.1} relies on Hrushovski's  theorem \cite[Corollary~1.2]{H}. In the literature, the reference for it is an unpublished manuscript of the author dated 1996 and entitled {\it The first order theory of Frobenius of the Frobenius automorphism}. Sometimes, a later reference to Conjecture 2 in {\it Chatzidakis, Z., Hrushovski, E.: Model Theory of Difference Fields, Trans. Am. Math. Soc. 351 (1999), 2997--3071} is given. One of the referees added another unpublished reference {\tt www.math.polytechnique.fr/\~{}giabicani/docs/memoireM2.pdf}.
 Since we do not need an effective version of Hrushovski's theorem, we refer to \cite{H} which is available on the  arXiv server. 

We profoundly thank Ehud Hurshovski for kindly sending us the latest version of this article. We thank Yves Andr\'e, Alexander Beilinson, Ph\`ung H\^o Hai and 
Alex Lubotzky for important discussions  directly or indirectly related to the subject of this article.
More specifically, we acknowledge discussions with Andr\'e on stratifications and liftings to characteristic 0, with Beilinson on $\sD$-modules, with Hai on the Tannaka formalism on the category of stratified bundles,  with Lubotzky on residually finite groups and Margulis rigidity. In addition, we thank Michel Raynaud for a very careful and friendly reading of the first version of our article. He helped us removing inaccuracies in Corollary \ref{cor3.16} and presenting the proof by separating what does not depend on the assumption that $X$ is geometrically simply connected (see  Theorem  \ref{thm3.14}). Furthermore, he directly contributed an illustration of  Theorem \ref{thm3.14} (see 
Proposition \ref{prop3.17}).  We also thank Burt Totaro for pointing out \cite{Mal} to us.

The second author thanks the 
 the SFB/TR45 at the University Duisburg-Essen for hospitality at several occasions during the last years.  
\section{Reduction to the case where the bundles are stable} 
Let $X$ be a smooth connected projective variety defined over a perfect field $k$ of characteristic $p>0$. \\[.1cm]
By Katz' equivalence of category between coherent $\sD_X$-modules which are $\sO_X$-coherent and stratified bundles \cite[Theorem~1.3]{G}, any stratified bundle $E=(E_n,\sigma_n)_{n\in \N}, \ \sigma_n: F^*E_{n+1}\simeq E_n$ of rank $r>0$ is isomorphic as a stratified bundle to one for which for all $n\ge 0$ and all
$ a\ge 0$, $(F^n)^{-1}E_{n+a}\subset E_a$ is  a subsheaf of abelian groups,  which  is a $(F^n)^{-1}\sO_X$-locally free module   of rank $r$.  From now on, we will always take a representative of a stratified bundle which has this property and we will drop the isomorphisms $\sigma_n$ from the notation. \\[.1cm]
Let $\sO_X(1)$ be an ample line bundle. For any bundle $E$ of rank $r>0$, one defines the {\it Hilbert polynomial of $E$ relative to $\sO_X(1)$} by $p_E(m)=\frac{1}{r} \chi\big(X, E(m)\big)\in \Q[m]$ for $m>>0$ large enough so that it is a polynomial, and where $E(m)=E\otimes_{\sO_X} \sO_X(m)$. Recall (\cite[section~0]{G1}) that $E$ is said to be $\chi$-stable ($\chi$-semistable) if for all subsheaves $U\subset E$ one has $p_U<p_E$ ($p_U\le p_E$). Here the order is defined by the values of the polynomials for $n$ large, and the rank of the (necessarily torsionfree) subsheaf $U$ is its generic rank. Recall that $E$ is said to be $\mu$-stable ($\mu$-semistable) if for all subsheaves $U\subset E$ on has $\mu(U)<\mu(E)$ ($\mu(U)\le \mu(E)$).  \\[.1cm]
We use the notation $CH^i(X)$ for the Chow group of codimension $i\ge 0$ cycles and denote by $\cdot$ the cup-product $CH^i(X)\times CH^j(X)\to CH^{i+j}(X)$.  If $X$ is connected and has dimension $d$, we denote by ${\rm deg}: CH^d(X)\to \Z$ the degree homomorphism.\\[.1cm]
With this notation,  the {\it slope} $\mu(E)$ of a torsionfree sheaf $E$ is defined by the formula $\mu(E)= \frac{1}{{\rm rank}(E)}{\rm deg} \big(c_1(E)\cdot \sO_X(1)^{d-1}\big)$, where $d={\rm dim}_k X$.
\begin{lem} \label{lem2.1}
 Let $X$ be a smooth connected projective variety of dimension $d\ge 1$ defined over a perfect field $k$ of characteristic $p>0$. Let $E=(E_n)_{n\in \N}$ be a stratified bundle. For any class $\xi\in CH^i(X)$, for all $n\ge 0$, and all $0\le i \le d-1$, one has ${\rm deg}\big(\xi\cdot \gamma_{d-i}(E_n)\big)=0 $, where $\gamma_{d-i}$ is any homogeneous polynomial  of degree $d-i$ with rational coefficients in the Chern classes.  
\end{lem}
 \begin{proof}
  One has ${\rm deg}\big( \xi\cdot c_{d-i}(E_n)\big)={\rm deg}\big( \xi\cdot c_{d-i}((F^a)^*E_{n+a})\big)$ for all $a\ge 0$. On the other hand, for any bundle $E$, one has $\gamma_{d-i}(F^*E)=p^{d-i} \gamma_{d-i}(E)$. Thus ${\rm deg}\big( \xi\cdot c_{d-i}(E_n)\big)\in \frac{1}{D}\Z$ is infinitely $p$-divisible, where $D$ is the bounded denominator $\in \N\setminus \{0\}$. Thus this is 0.
 \end{proof}
\begin{cor}  \label{cor2.2}
 Let $X$ be a smooth connected projective variety of dimension $d\ge 1$ defined over a perfect field $k$ of characteristic $p>0$. Let $E=(E_n)_{n\in \N}$ be a stratified bundle. Then $p_{E_n}=p_{\sO_X}$ and $\mu(E_n)=0$ for any $n\ge 0$. 
\end{cor}

\begin{prop} \label{prop2.3}
  Let $X$ be a smooth connected projective variety defined over a perfect field $k$ of characteristic $p>0$. For any stratified bundle $E=(E_n)_{n\in \N}$, there is a $n_0\in \N$ 
 such that the stratified bundle $E(n_0)=(E_n)_{n\ge n_0, n\in \N}$ is a successive extension of stratified bundles $U=(U_n)_{n\in \N}$ with the property that all $U_n$ for $ n\in \N$ are $\mu$-stable of slope $0$. In particular, all $U_n$ for $n\in \N$ are  $\chi$-stable bundles of Hilbert polynomial $p_{\sO_X}$. 
\end{prop}
\begin{proof}
Since $\mu$-stability implies $\chi$-stability (\cite[Lemma~1.2.13]{HL}), it is enough to prove the  proposition with $\mu$-stability. We first show that for $n_0$ large enough, $E_n, n\ge n_0$ is $\mu$-semistable (see \cite[Lemma~2.2]{BK}). Let $U_n$ be the a nontrivial subsheaf of $E_n$. Assume $\mu(U_n)>0=\mu(E_n)$ (Corollary \ref{cor2.2}). Since $(F^n)^*(U_n)\subset E_0$,  $\mu((F^n)^*(U_n))=p^n \mu(U_n)$ is bounded by $\mu_{{\rm max}}(E_0)$, the slope of the maximal destabilizing subsheaf of $E_0$. One concludes that there is a $n_0\ge 0$ such that for all  $n\ge n_0$, one has  
$\mu_n(U_n)\le 0$. Thus $E_n$ is $\mu$-semistable for $n\ge n_0$. \\[.1cm]
To show the proposition, we may now assume that $n_0=0$, that is all $E_n$ are $\mu$-semistable of slope $0$. 
Let $U_n\subset E_n$ be the socle of $E_n$, that is the maximal nontrivial subsheaf which is $\mu$-polystable of slope $0$. Then $(F^n)^*(U_n)\subset E_0$ has still slope 0, thus has to lie in $U_0$. This yields a decreasing sequence $\ldots \subset (F^{n+1})^*(U_{n+1})\subset (F^n)^*(U_n)\subset \ldots \subset  E_0$, which has to be stationary for $n$ large. Thus there is a $n_1 \ge 0$ such that for all $n\ge n_1$, 
$F^*U_{n+1}=U_n$. So $(U_{n_1}, U_{n_1+1}, \ldots )\subset E(n_1)$ is a substratified sheaf. Thus it is a substratified bundle, as $\sO_X$-coherent $\sD_X$-coherent modules are locally free (\cite[Lemma~6]{dS}).\\[.1cm]
So we may assume $n_1=0$.   Write $U_n=\oplus_{b=1}^{a(n)} S_n^b$ where $S_n^b$ is $\mu$-stable of slope $0$. Then one has an exact sequence  $0\to F^*S_{n+1}^b\cap S_n^c \to  F^*S_{n+1}^b\oplus  S_n^c \to  F^*S_{n+1}^b + S_n^c\to 0$ where the sum on the right is taken in $U_n$.  Thus $\mu(
F^*S_{n+1}^b + S_n^c)\le 0$ and $\mu( F^*S_{n+1}^b\cap S_n^c)\le 0$, and we conclude that  $F^*S_{n+1}^b\cap S_n^c$ is either $S^c_n$ or else is equal to $0$. We conclude that $U=(U_n)_{n\in \N}\subset E=(E_n)_{n\in \N}$ is a direct sum $U=\oplus_{b=1}^aU^b$ where $U^b=(U_n^b)_{n\in \N}$ is a stratified bundle which has the property that $U_n^b$ is $\mu$-stable of slope $0$.\\[.1cm]
We now finish the proof using that the category of stratified bundles is abelian: we replace $E$ by $E/U$, which has a strictly lower rank, and we redo the argument. 
\end{proof}
\begin{prop} \label{prop2.4}
Let $X$ be a smooth projective variety defined over a perfect field $k$ of characteristic $p>0$. Let $x\to X$ be a geometric point. If  $\pi^{{\rm \et}}_1(X\otimes_k {\bar k}, x)$ has no  quotient isomorphic to $\Z/p$, 
 then stratified bundles which are successive extensions of the trivial stratified bundle by itself are trivial. 
\end{prop}
\begin{proof} We assume $k=\bar k$. 
 We  have to show that a stratified bundle $E=(E_n)_{n\in \N}$ which is a successive extension 
of ${\mathbb I}:=(\sO_X, \sO_X,\ldots)$ by itself is trivial. 
 By dos Santos' theorem  \cite[(9)]{dS}, the isomorphism class of $E$ lies in the projective system 
$\varprojlim_n H^1(X, \sO_X)$ where the transition maps are the pullback maps $F^*: H^1(X, \sO_X)\to H^1(X, \sO_X)$ via the absolute Frobenius. On the other hand, since $k$ is perfect, one has the decomposition  $H^1(X, \sO_X)=
H^1(X, \sO_X)_{{\rm ss}}\oplus H^1(X, \sO_X)_{{\rm nilp}}$ where $H^1(X, \sO_X)_{{\rm ss}}=H^1(X, \Z/p\Z)\otimes_{\F_p} k$ and $H^1(X, \sO_X)_{{\rm nilp}} \subset H^1(X, \sO_X)$ is defined as the $k$-subvectorspace of classes on which $F^*$ is nilpotent. Since $H^1(X, \Z/p\Z)={\rm Hom}(\pi_1(X), \Z/p\Z)$, this group is 0 by the assumption.
 Since $H^1(X, \sO_X)$ is a finite dimensional vectorspace, and $F$ is semilinear, there is a $N\in \N\setminus \{0\}$ such that $F^{N*}$ annihilates $H^1(X, \sO_X)_{{\rm nilp}}$. Thus we conclude $\varprojlim_n H^1(X, \sO_X)=\varprojlim_n H^1(X, \sO_X)_{{\rm nilp}}=0$. Thus the stratified bundle $E$ is trivial. This finishes the proof over $k=\bar k$. In general, this shows that if $E$ is a stratified bundle on $X$ which is a successive extension of the trivial stratified bundle by itself, then $E\otimes_k \bar k$ is a trivial stratified bundle on $X\otimes_k \bar k$. Thus for any $n\ge 0$, $H^0(X, E_n)\otimes_k \sO_X \to E_n$ is an isomorphism after tensoring with $\bar k$ over $k$, thus it is an isomorphism. Thus $E_n$ is trivial for all $n\ge 0$, so $E$ is trivial. This finishes the proof in general.  
\end{proof}

\section{The proof of the main theorem and its corollary}
Let $X$ be a smooth projective connected variety defined over a perfect  field $k$ of characteristic $p>0$. We fix an ample line bundle $\sO_X(1)$, a rank $r>1$ and consider the quasi-projective moduli scheme $M$ of $\chi$-{\it stable} torsionfree sheaves with Hilbert polynomial $p_{\sO_X}$ and rank $r$, as defined by Gieseker \cite[Theorem~0.2]{G1} in dimension 2 and 
and Langer \cite[Theorem~4.1]{L2} in any dimension.\\[.1cm] Even if one can reduce the statement of Theorem \ref{thm1.1} to the dimension 2 case by a Lefschetz type argument on stratified bundles, we will need the strength of Langer's theorem.
Let us recall from {\it loc. cit.} that if $S$ is a smooth absolutely connected affine variety over a finite field $\F_q$ such that $\F_q(S)\subset k$, and if $X_S\to S$ is a smooth projective model of $X$, then there is a quasiprojective scheme $M_S\to S$ which universally corepresents the functor of families of $\chi$-stable torsionfree sheaves 
  of rank $r$ and Hilbert polynomial $p_{\sO_X}$ on geometric fibers.
This concept  is due to Simpson \cite[p.~60]{Simp}. The scheme $M_S$ has several properties. It represents the \'etale sheaf  associated to the moduli functor. If $s\in S$ is any closed point, then $M_S\times_S s$ is $M_s$. If $u\in M_S$ is any closed point above the closed point $s\in S$, there is a $\chi$-stable torsionfree sheaf  $E$ on $X_S\times_S s$ with moduli point $u$.  
  \\[.1cm]
We will use the notation $[E]\in M$ to indicate the closed point in $M$ which represents the stable bundle $E$. As is well known \cite[Theorem~1]{Gstable}, even if $E$ is stable, the bundle $F^*(E)$ need not be stable. On the other hand, stability is an open condition. Thus there is an open subscheme $M^0\subset M$  such that $F^*(E)$ is stable for all $[E]\in M^0$.
\begin{defn} \label{defn3.1}
 The morphism   $\xymatrix{ V: M^0\to M}$ is called the {\it Verschiebung}.
\end{defn}

\begin{defn} \label{defn3.2} 

 We define the locus $M^s\subset M$ consisting of points $[E]\in M$ which are $\mu$-stable of slope $0$ such  that there is a statified bundle $(E_n)_{n\in \N}$ with $E=E_0$. (The upper script $s$ here stands for stratified.)

\end{defn}
\noindent This is the key definition of the article due to the following remark. 

\begin{rmk} \label{rmk3.3}
All the bundles $E_n, n\ge 0$ in Definition \ref{defn3.2} have the property that they are $\mu$-stable of slope $0$ by the computation of Proposition \ref{prop2.3}: $\mu((F^n)^*U_n)=p^n\mu(U_n)$, hence $\mu(U_n)<0$. Thus a point $[E] \in M^s$ defines a sequence of points $[E_n]_{n\in \N, n\ge 0}\in M^s$ such that $(E_n)_{n\in \N}$ is a stratified bundle.
\end{rmk}
\noindent We now define various closed subschemes of $M$ using $M^s$. 
\begin{defn} \label{defn3.4} Let $E=(E_n)_{n\in \N}$ be a stratified bundle with $E_0\in M^s$ (thus $E_n\in M^s$ for all $n\ge 0$).
\begin{itemize}
\item[1)]  We define $A(E)$ to be the Zariski closure of the locus $\{E_n, n\ge 0\}$ in $M$. 
\item[2)] We define $N(E)=\cap_{n\ge 0} A(E(n))$ where $E(n)=(E_{n+m, m\ge 0})$.

\end{itemize}  \noindent

Closed subschemes of $M$ of type $N(E)$ as in 2)  are called {\it subschemes of} $M$ {\it spanned by stratifications}. 
\end{defn}
\begin{rmks} \label{rmk3.5} 
\begin{itemize}
\item[1)] As $M^s$ is defined by its $k$-points, any   subscheme of $M$ spanned by stratifications lies in $M_{{\rm red}}$. 
\item[2)] Let $E=(E_n)_{n\in \N}$ be a stratified bundle with $E_n\in M^s$ for all $n\ge 0$. As $A(E(n+1))\subset A(E(n))$ are closed subschemes of $M_{{\rm red}}$, by the noetherian property there is a $n_0\ge 0$ such that $A(E(n))=A(E(n_0))=N(E)$ for all $n\ge n_0$.  Thus  there is a hierarchy in the definition:
subschemes of  type $N(E)$ are  of type $A(E)$, and 
to show that $M^s$ is empty is equivalent to showing that subschemes of $M$ spanned by stratifications are empty.

\end{itemize}
\end{rmks}
\begin{defn}  \label{defn3.6}
 A closed subscheme  $N\subset M_{{\rm red}}$  is called {\it Verschiebung divisible} if $
\xymatrix{ V|_{N}: N\ar@{.>}[r] & N}$ is a rational  map which is dominant on all the components of $N$. We denote by $N^1\subset N$ the dense locus on which $V|_N$ is defined. 
 
\end{defn}

\begin{lem} \label{lem3.7} Let $N$ be a Verschiebung divisible subscheme of $M$. 
For all natural numbers $a\in \N\setminus \{0\}$, the composite $V|_{N}^a:=V|_{N} \circ \ldots \circ V|_{N}$ ($a$-times) is a dominant rational map $\xymatrix{ V|_{N}^a: N\ar@{.>}[r] & N}$ and there is a natural number $a \in \N\setminus \{0\}$ such that $V|_{N}^a$ stabilizes all the irreducible components of $N$. 
\end{lem}
\begin{proof}
 Since $V|_N$ is a rational dominant self-map, the iterate maps $V^a|_N$ are rational dominant self-maps as well. On the other hand, the image of an irreducible component by a rational map lies in an irreducible component. So by a counting argument, a rational dominant self-map permutes the irreducible components. Thus an iterate of the map stabilizes the components.  
\end{proof}

\begin{lem} \label{lem3.8}
 Subschemes of $M$ which are spanned by stratifications are Verschiebung divisible. 
\end{lem}
\begin{proof} If $N$ is empty, there is nothing to prove. Else $N$ is constructed as the Zariski closure of family of points $S=\{e_n, n\ge n_0\}$ in $M$ which have the property that for all $n\ge n_0$, there is a point $e_{n+1}\in N$ with $V(e_{n+1})=e_n$, and which in addition have the property that $ N$ is also the Zariski closure of the family of points $S[m]=\{e_n, n\ge m+ n_0\}$ for any $m\ge 0$.
 Taking $m=1$, one has $V(S)=S[1]$.
Let $N^1\subset N$ be the locus on which $V$ is defined and $V(N^1)\subset N$.  Then $N^1$ contains $S[1]$. Thus 
$V|_{N}: N^1\to N$ is dominant on $N$ and is dense in $N$. This finishes the proof.

\end{proof}
Recall that if $Z$ is any scheme of finite type defined over $k$, and  $S$  is a smooth affine variety, which is defined over a finite field $\F_{q}$, is  geometrically irreducible over $\F_{q}$, then a {\it model} $Z_S/S$ is a $S$-flat scheme such that $Z_S\otimes_S k=Z$. \\[.1cm]
Let
 $S_0$  be a smooth affine variety, which is defined over a finite field $\F_{q_0}$, is  geometrically irreducible over $\F_{q_0}$, such that 
\begin{itemize}
 \item[i)] $X/k$ has a smooth projective model $X_{S_0}\to S_0$,
\item[ii)]  the absolute Frobenius $F=F_{X}: X\to X$ over the Frobenius $F_k: \Spec k\to \Spec k$ has a model $F_{X_{S_0}}: X_{S_0} \to X_{S_0}$ over the absolute Frobenius $F_{S_0}: S_0\to S_0$.
\end{itemize}
 By Langer's theorem \cite[Theorem~4.1]{L}, there is then a quasiprojective model $M_{S_0}\to S_0$ of $M/k$ which universally corepresents the functor of $\chi$-stable torsionfree sheaves of rank $r$ and Hilbert polynomial $p_{\sO_X}$. 
\begin{defn} \label{defn3.9} Let
$S\to S_0$ by a morphism, such that $S$ is a smooth affine variety, geometrically irreducible over a finite extension $\F_q$ of $\F_{q_0}$, with $\F_{q_0}(S_0)\subset \F_q(S)\subset k$. Let  $X_S, M_S$ the base changed varieties $X_{S_0}, M_{S_0}$. 
Let $N$ be a Verschiebung divisible subscheme  of $M_{{\rm red}}$.  Then a model $N_S$ of $N$  is called {\it a good model} if 
\begin{itemize}
\item[a)] all irreducible components $N_i\subset  N\subset M_{{\rm red}}$ of $N$, for  $ i=1,\ldots, \rho$,  have a model $N_{iS}\subset M_S$ over $S$,
\item[b)] $V|_N: N^1\to N$ has a model 
$V_S: N^1_S\to  N_S$, where $N_S=\cup_{i=1}^\rho N_{iS}\subset M_S$
and $N^1_S\subset N_S$ is dense.
\end{itemize}
\end{defn}
\noindent 
Note, for some closed point $s\to S$, $N_{iS}\times_S s$ might be reducible.\\[.1cm]
For a closed point $s\to S_0$, we denote by $M_{S_0}\times_{S_0} s, \ X_{S_0}\times_{S_0} s=X_s, \  F_{S_0}\times_{S_0} s$ the reductions to $s$. By the universal corepresentability, one has $M_{S_0}\times_{S_0} s= M_s$, Langer's  moduli scheme of $\chi$-stable torsionfree sheaves of rank $r$ with Hilbert polynomial $p_{\sO_X}= p_{\sO_{X_s}}$. By definition,  
$F_{S_0}\times_{S_0} s$  is the absolute Frobenius $F_s$.  We denote by $V_s$ the Verschiebung on $M_s$ and by $V_{S_0}\times_S s$ the reduction of $V_S$ to $s$ on the locus where it is defined. Let $M^0_{S}\subset M_{S}$ be a model of $M^0$, with $S\to S_0$, with $S$ smooth affine over a finite extension $\F_q$ of $\F_{q_0}$. Then by definition, 
$V_{S}\times_{S}s$ is defined on $M^0_{S}\times_{S}s$. 
\begin{lem} \label{lem3.10} 
One has $V_s= V_{S}\times_S s$ on  $M^0_{S}\times_{S}s$.
\end{lem}
\begin{proof}
Let $t\to M^{0}_{S}\times_S s$ be a closed point. Since $\F_q$ is perfect, $t\to s$ is \'etale. There is a morphism of schemes $T\to M^{0}_{S}$, 
such that $T$ is irreducible over $\F_q$, the composite $T\to M^{0}_S \to S$ is finite and \'etale onto its image which is a neighbourhood containing $s$,  $t\in M^0_S$ is the image of $t_0\in T$,  $T\to M^0_{S}$ is a closed embedding in a Zariski open containing $t_0$. In particular, one has a factorization $\F_q(S)\subset \F_q(T)\subset k$. 
By the universal corepresentability, there is a torsionfree sheaf $E_T$ over $X_T$, 
such that $E_T\times_T k$  is $\chi$-stable of rank $r$ and with Hilbert polynomial $p_{\sO_X}$. By definition of $M^{0}_{S}$, $[E_T\times_T k]
\in M^{0}$, thus $F^*(E_T\times_T k)$ is $\chi$-stable of Hilbert polynomial $p_{\sO_X}$ and rank $r$. The sheaf
$F^*(E_T\times_T k)$ has a model which we explain now. The pullback of the absolute Frobenius $F_{S}: S\to S$ by $T\to S$ is $F_T: T\to T$, as 
$T\to S$ is \'etale. Thus the pullback of the absolute Frobenius $F_{X_{S}}: X_{S}\to X_{S}$ over $F_{S}$ by $T\to S$ is the absolute Frobenius $F_{X_{T}}: X_{T} \to X_{T}$. This defines the model $F_{X_{T}}^*E_{T}$ of $F^*(E_T\times_T k)$.
We conclude that  $F_{X_{T}}^*(E_{T})\times_T t_0= F_{X_{t_0}}^*E_{t_0}$. By definition of $t_0$, $F_{X_{t_0}}^*E_{t_0}=F_{X_t}^*E_t$. By definition, $[F_{X_t}^*E_t] \in M_t$. We conclude
 $[F_{X_t}^*E_t]=V_t([E_t])$. 
\end{proof}
\begin{cor} \label{cor3.11}
Let $N$ be a Verschiebung divisible subscheme of $M_{{\rm red}}$, and let
$N_S$ be a good model.  
Let $a\in \N\setminus\{0\}$ such that $V|_{N}^a$ stabilizes the irreducible components of $N$ as in Lemma
\ref{lem3.7}. Then 
\begin{itemize}
 \item[i)] 
there is a nontrivial open subscheme $T\subset S$ such that for all closed points $s\in T$, and for all $i\in \{1,\ldots, \rho\}$, $V_T|_{N_{iT}}^a \times_T s$ is a dominant rational self map of $N_{iT}\times_T s$;
\item[ii)]  for any closed point $s\in T$, there is a dense open subscheme of 
$N_{iT}\times_T s$ 
on which    $V_T|_{N_{iT}}^a \times_T s$  is well defined and is equal to   $V_s^a=V_s\circ \ldots \circ V_s$ ($a$-times). \end{itemize}
\end{cor}
\begin{proof}
 By definition, the restriction of $V$  on $N$ is a rational dominant self-map. As $V_S$ is is a good model, $V_S: N_S^1\to N_S$ is a well defined dominant map, and $N^1_S\subset N_S$ is dense.  Thus there is a nontrivial open $S_a\subset S$ and a dense subscheme $N^a_{S_a}\subset N^1_{S}$, such that the composite $V_{S_a}^a=V_{S_a}\circ \ldots V_{S_a}$ ($a$-times) is well defined and dominant. So there is a nontrivial open $T\subset S_a$ such that for all closed points $t\in T$, the restriction $V_{T}^a\times_T t: N^a_T\times_T t\to N_T\times_T t=N_t$ is well defined and is dominant.  By the choice of $a$, $V_T^a$ respects the components $V_{iT}$. This shows i). As for ii), this is then a direct consequence of Lemma \ref{lem3.10}.
This finishes the proof. 
\end{proof}
\begin{defn} \label{defn3.12}
Let $X$ be a smooth projective variety defined over a perfect   field $k$ of characteristic $p>0$.
   Let $M$ be the moduli of $\chi$-stable  sheaves of rank $r$ and Hilbert polynomial $p_{\sO_X}$. 
\begin{itemize}
\item[1)]A closed point $[E]\in M$ is called a {\it torsion point} if there is a $m\in \N\setminus \{0\}$ such that $(F^m)^*E\cong E$. 
\item[2)] Let  $S$ be a smooth affine variety defined over $\F_q$ such that $X$ has a smooth projective model $X_S$, and let $M_S$  Langer's quasi-projective moduli $S$-scheme. Let $u\in M_S$ be a closed point, mapping to the closed point $s\in S$.   Then $u$ is said to be  a {\it torsion point} if  $u$ is a torsion point viewed as a closed point in $M_s$.  
\end{itemize}
\end{defn}
\begin{rmk} \label{rmk3.13}
 If  $u\in M_S$ is a torsion point mapping to $s\in S$, then the torsionfree sheaf $E$ on $X_s$ it corresponds to  has a stratification $E=E_0, E_1=(F^{N-1})^*E,\ldots , E_N=E, E_{N+1}=E_1, \ldots$. It follows by Katz' theorem \cite[Theorem~1.3]{G} that $E=(E_n)_{n\in \N}$ is a  stratified bundle (that is, all the $E_n$ are locally free).   Furthermore, all $E_n$ are stable of slope $0$, so all $E_n$ define modular points in $M_s$.
\end{rmk}

\begin{thm}  \label{thm3.14}
Let $X$ be a smooth projective variety defined over a perfect field $k$ of characteristic $p>0$. Then $M$ be the moduli of $\chi$-stable torsionfree sheaves of rank $r$. Let $N\subset M_{{\rm red}}$ be a Verschiebung divisible subscheme, and let $N_S$ be a good model of $N$. Then the torsion points of $N_S$ are dense in $N_S$.  
\end{thm}
\begin{proof}
Our goal is to show that if $Z\subset N_S$ is any strict closed subscheme, then $N_S\setminus Z$ contains torsion points. So it is enough to show that $N_T\setminus Z$ contains torsion points, where $T$ is defined in Corollary 
\ref{cor3.11}. 
   Let $s_0\in T$ be a  closed point  and  let $q(s_0)$ be the cardinality of the residue field $\kappa(s_0)$. We want to show that there are  torsion points in $N_{T}\times_T s_0\setminus (Z\cap N_{T}\times_T s_0)$. \\[.2cm]
We apply Hrushovski's theorem \cite[Corollary~1.2]{H}:  let $i \in \{1,\ldots, \rho\}$. This fixes the component $N_{iT}$ we  consider. 
 We define $ \kappa$ to be a finite extension of $\kappa(s_0)$ of cardinality $q$,  so that  the irreducible components of $N_{is}:=N_{iT}\otimes_T \kappa $ are defined over $\kappa$. Here $s=s_0\otimes_{\kappa(s_0)} \kappa$. We pick one such irreducible component of maximal dimension, say $Y'$. Since $V^a_s$  stabilizes $N_{is}$, the same argument as in Lemma \ref{lem3.7} implies that $V^{ab}_s $  stabilizes  $Y'$ for some $b\in \N\setminus \{0\}$. 
We choose 
an open dense affine subvariety $Y\subset Y'$. Then $Y$ is irreducible over $\F_{q}$.  Let $\Gamma \subset Y\times_s Y$ be the Zariski closure of the 
graph of $\Psi:=V_s^{ab}$ where it is well defined.  Then the first projection $\Gamma\to Y$ is birational (and therefore dominant) and the second projection $\Gamma  \to Y$ is dominant as 
$\Psi$
is dominant. Let $\Phi_{q}: Y\otimes_{\kappa} \bar \F_{q} \to  Y\otimes_{\kappa} \bar \F_{q}$ be the Frobenius raising coordinates to the $q$-th power. Hrushovski's theorem {\it loc. cit.} asserts that for any proper closed subvariety $W\subset Y$, there is a closed point $u\in Y\setminus W$, thus that for a suited $m\in \N\setminus \{0\}$ large enough, $(u, \Phi_{q}^m(u))$ is a closed point of $\Gamma$. Taking $W$ to contain both the locus on which $\Psi$  is not defined, and $Z\cap Y$, 
we obtain that 
$(u, v)\in \Gamma$ if and only if $v=\Psi(u)$, thus $\Phi_{q}^m(u)=\Psi(u)$, and $u\notin Z\cap Y$.  On the other hand, one has $\Phi_{q} \circ V_s(u)= V_s\circ \Phi_{q} (u)$ as $V_s$ is defined over $\kappa$. Thus from $ \Phi_{q}^m(u)=V_s^{ab}(u)$, we deduce that $ \Phi_{q}^{mc}(u)=V_s^{abc}(u)$ for all $c\in \N\setminus \{0\}$. As $u \in T$ is a closed point, there is a  $c$ such that $\Phi_{q}^{mc}(u)=u$. We deduce that for all proper 
closed subvarieties $W\subset Y$
containing the indeterminacy locus of $\Psi$ and $Z\cap Y$, there is a closed point $u\in Y\setminus W$ and a natural number $d\in \N\setminus \{0\}$, such that  $u= V_s^{d}(u)$. 
Thus $u$ is a torsion point. This finishes the proof.

\end{proof}

\begin{thm} \label{thm3.15}
 Let $X$ be a smooth connected projective variety defined over a perfect field $k$ of characteristic $p>0$. Let $\bar x\to X$ be a geometric point. If $\pi_1^{{\rm \et}}(X \otimes_k \bar k,\bar x)=\{1\}$,  there are no nontrivial stratified bundles. 
\end{thm}
\begin{proof} 
We consider rank $r\ge 2$ stratified bundles, as for rank 1, we gave the proof in the introduction. We want to show that $M^s$ is empty. This is equivalent to saying that the Zariski closure of $M^s$ in $M$ is empty, and is also equivalent to saying that any  subscheme $N$ of $M_{{\rm red}}$ spanned by stratifications is empty (see Remarks \ref{rmk3.5} 2)). Let $N_S$ be a good model of $N$. 
By Lemma \ref{lem3.8} together with Theorem \ref{thm3.14},  torsion points 
are dense in $N_S$. By Remark \ref{rmk3.13}, a torsion point $u\in N_T$ mapping to $s\in S$ represents in particular a vector bundle $E$ on $X_s$.  
By the theorem of Lange-Stuhler \cite[Satz~1.4]{LS}, 
  there is a (noncommutative) geometrically connected   \'etale finite  covering $h:Z\to X_s$,  such that $h^*E$ is trivial (beware that the other statement of {\it loc.cit}, asserting that for any vector bundle $E$,  an \'etale finite covering $h$ 
which trivializes $h^*E$ exists if and only if $ (F^m)^*E\cong E$ for some $m\in \N\setminus \{0\}$ is not correct,  
 although there are many ways to correct it). On the other hand, since $X_S\to S$ has good reduction and is proper, the specialization map $\pi_1(X_{\bar k})\to \pi_1(X_s\times_s \Spec \bar \F_q)$ is surjective
(\cite[Expos\'e~X,~Th\'eor\`eme~3.8]{SGA1}). The assumption implies then that  
$\pi_1(X_s\times_s \Spec \bar \F_q)=\{1\}$,
thus $h$ is the identity, and $E \simeq \oplus_1^r \sO_{X_s}$.  But $E$ has to be stable of rank $r\ge 2$. This is impossible. So we conclude that $N_S$ is empty, thus $N$ is empty. So there are no stratifed bundles  $E=(E_n)_{n\in N}$ with  $E_n$ $\mu$-stable of rank $r$ and slope zero, except $E=\mathbb{I}$. 
By Proposition \ref{prop2.3}, for any stratified bundle, $E(n_0)$ is a successive extension of $\mathbb{I}$ by itself for $n_0$ large enough. 
We apply Proposition \ref{prop2.4} to finish the proof over $k=\bar k$. If $k$ is perfect and not algebraically closed,    as already noticed in the proof of Proposition \ref{prop2.4}, a stratified bundle $E$ is trivial if and only if $E\otimes_k \bar k$ is trivial. This finishes the proof in general.
\end{proof}
In the remaining part of this section, we illustrate the theorem with two examples. The second one is due to M. Raynaud. \\[.1cm]
Let $X$ be a smooth projective variety defined over an algebraically closed field $k=\bar k$ of characteristic $p>0$. We know that if $k\neq \bar\F_p$, and if $H^1(X, \sO_X)_{\rm{ss}}$ has $k$-dimension at least $\ge 2$, one easily constructs an infinite family of extensions of $\mathbb{I}$ by itself in ${\sf Strat}(X)$ (see \cite[Proposition~2.9]{BK}). Furthermore, there are stratified bundles which are not semistable, so they can't 
be trivialized after a finite \'etale covering (\cite{G}). Thus the assumption on the finiteness of $\pi_1^{{\rm \et}}(X \otimes_k \bar k,\bar x)$ in the next corollary is really necessary. 
\begin{cor}  \label{cor3.16}
 Let $X$ be a smooth connected projective variety defined over an
algebraically closed field $k$ of characteristic $p>0$. Let $ x \in X(k)$ be a rational point. Assume $\pi_1^{{\rm \et}}(X, x)$ is finite. Then
\begin{itemize}
\item[i)] the surjective homomorphism 
$\pi^{{\rm strat}}(X,  x)\to 
\pi_1^{{\rm \et}}(X, x)$ induces an isomorphism on irreducible representations; 
\item[ii)] if  $\pi_1^{{\rm \et}}(X, x)$ has order prime to $p$, then the surjective homomorphism $\pi^{{\rm strat}}(X,  x)\to 
\pi_1^{{\rm \et}}(X, x)$ is an isomorphism; in particular, every stratified bundle is a direct sum of irreducible ones.

\end{itemize}
 
\end{cor}
\begin{proof}
Let $h: Y\to X$ be the universal cover based at $x$, so it is a Galois cover under $\pi_1^{{\rm \et}}(X,  x)$.
Let $\mathbb{I}_Y$ be the trivial stratified bundle on $Y$, and $\mathbb{E}=h_*\mathbb{I}_Y$ be its direct image. It is an object of ${\sf Strat}(X)$ which comes from the regular representation $k[\pi_1^{{\rm \et}}(X, x)]$ of $\pi_1^{{\rm \et}}(X, x)$.
 If $E=(E_n)_{n\in \N}$ is a stratified object, then by Theorem \ref{thm3.15}, $h^*E$ is trivial, thus $E\subset 
H^0(Y, h^*E_0)\otimes \E$. We first show i). If $E$ is irreducible as a stratified bundle, there is a projection of $H^0(Y, h^*E_0)$ to a $k$-line $\ell$ such that $E$ is still injective in $\ell\otimes_k \mathbb{E}$. Thus $E\subset \ell\otimes_k \mathbb{E}$ is a subrepresentation of $\pi^{{\rm strat}}(X,  x)$.
Since 
$\pi^{{\rm strat}}(X,  x)\to 
\pi_1^{{\rm \et}}(X, x)$ is surjective, $E \subset \ell\otimes_k \mathbb{E}$ 
is a subrepresentation of $\pi_1^{{\rm \et}}(X, x)$. This shows i). As for ii), $h$ then has degree prime to $p$, thus $\mathbb{E}$ is a direct sum of irreducible representations of $\pi^{{\rm \et}}(X, x)$, thus 
$E \subset H^0(Y, h^*E_0)\otimes_k \mathbb{E}$ as well, and is in particular 
a subrepresentation of $\pi^{{\rm \et}}(X,  x)$.
\end{proof}
We now illustrate Theorem \ref{thm3.14} in rank $2$ over $k=\bar \F_p$. To this aim, recall that if $E$ is a stratified bundle over $X$ projective smooth over $k=\bar k$, and if $x\in X(k)$, the Tannaka $k$-group ${\rm Aut}^{\otimes }(\langle E\rangle, x)\subset GL(E_0|_x)$ is also called the {\it monodromy group}. Recall further that if $[E]$ is a torsion point (see Definition \ref{defn3.12}) in $M$, then its monodromy group is in fact a finite quotient of $\pi^{{\rm \et}}(X,x)$, as $E$ is trivialized on the Lange-Stuhler finite \'etale covering $h$ ({\it loc. cit.}), so the monodromy group is a quotient of the Galois group of $h$. 
\begin{prop}[M.~Raynaud] \label{prop3.17}
 Let $X$ be a smooth projective variety defined over $k=\bar \F_p$. Let $E=(E_n)_{n\in \N}$ be a rank 2 stratified bundle with $E_n\in M$, the moduli of $\chi$-stable rank 2 torsionfree sheaves,  and with ${\rm det}(E)=\mathbb{I}$ (i.e. ${\rm det}(E_n)=\sO_X$ for all $n\ge 0$).  Let $N(E)$ be defined in Definition \ref{defn3.4}. We assume $N(E)$ irreducible of dimension $>0$.  Then
\begin{itemize}
\item[i)] either there is a dense subset of torsion points $a_i \in N(E)$ with monodromy group $SL(2, k_i) \subset GL(E_0|_x)$ where $k_i$ are finite subfields of $k$ of increasing order;
\item[ii)] or there is a dense subset of torsion points $a_i\in N(E)$ with monodromy group a dihedral group $D_{n_i}$ of order $2n_i$ with  increasing $n_i$.

\end{itemize}
One can characterize geometrically the second case: it happens precisely when there is an \'etale degree 2 covering $h: Y\to X$ and a rank one stratified bundle $L$ on $Y$ such that $h_*L=E$, or, equivalently, such that $h^*E$ becomes reducible: $h^*E\cong L\oplus L^\sigma$, where $\langle \sigma\rangle  $ is the Galois group of $h$ and  $L^\sigma$ is the Galois translate of $L$. 
\end{prop}
\begin{proof}
 Assume there if a degree 2 \'etale covering $h: Y\to X$ and a rank 1 stratified bundle $L$ on $Y$ such that $L\subset h^*E$.  Then $L^\sigma\subset h^*E$ and thus $h^*E \cong L \oplus L^\sigma$. Since ${\rm det}(E)=\mathbb{I}$, one has $L^\sigma\cong L^{-1}$. One has in particular  $h_*L_n=E_n$ for all $n\ge 0$. All $L_n$ lie in $\Pic^\tau(Y)$. 
We define $N(L)\subset \Pic^\tau(Y)$ as in Definition \ref{defn3.4}. Then the morphism $\Pic^\tau(Y)\to M, \sL\mapsto h_*\sL$ sends $N(L)$ to a closed subscheme of $M_{{\rm red}}$, as $\Pic^\tau(Y)$ is proper, which is contained in $N(E)$ by construction, but  which contains all the $E_n=h_*L_n$. Thus $h_*N(L)=N(E)$. The subscheme  $N(L)$ is the closure of an infinite family of torsion points $t_m$. Since $h_*t_m$ lies in $N(E)$, and since  the locus of $M$ on which the determinant is $\sO_X$ is closed, all points of $N(E)$ have determinant equal to $\sO_X$. In particular, one has  ${\rm det}(h_*(t_m))=\sO_X$. This implies that $t_m^\sigma\cong t_m^{-1}$ as  ${\rm det}(h^*h_*t_m) \cong {\rm det}(t_m\oplus t_m^\sigma)=\sO_Y$. 
 Let $y\in Y(k)$ mapping to $x\in X(k)$. The exact sequence $1\to \pi_1^{{\rm \et}}(Y, y)\to  \pi_1^{{\rm \et}}(X, x) \xrightarrow{h} \Z/2\to 0$ induces by pushout an exact sequence $1\to H\to D \xrightarrow{h} \Z/2\to 0$ for any finite quotient $\pi_1^{{\rm \et}}(Y, y)\surj H$. Taking for $H$ the cyclic quotients $\Z/n(m)\Z$ corresponding to the $t_m$ yields a stratified line bundle $\sL_m$ on $Y$, such that the generator $\sigma \in \pi_1^{\rm \et}(X,x)$ of  the automorphism group of $Y$ over $X$   acts via $x\mapsto -x$ on $\Z/n(m)\Z$. Thus $D$ is the dihedral group $D_{n(m)}$. Summarizing: if there is a $h: Y\to X$ such that $h^*E$ becomes strictly semistable, then points with monodromy group a dihedral group $D_{n(m)}$ of increasing order
$n(m)$  are dense in $N(E)$. (Note we do not need the irreducibility of $N(E)$ for this point). \\[.1cm]
Vice-versa, assume the points $\delta_m$ with dihedral monodromy $D_{n(m)}$ are dense in $N(E)$. Since $\pi^{{\rm \et}}(X,x)$ has only finitely many $\Z/2\Z$ quotients, there is an infinite sequence of such points $\delta_m$ such that the induced quotient $\pi^{{\rm \et}}(X,x) \to \Z/2\Z$ is fixed. Let $h: Y\to X$ be the corresponding covering and $y\in Y(k)$ mapping to $x$. The representation $\pi^{{\rm \et}}(X,x)\surj D_{n(m)}$ induces a representation $\pi^{{\rm \et}}(Y,y)\surj \Z/n(m)\Z$, defining  a stratified rank 1 bundle $\sL_m$ on $Y$, such that  $\sL_m^\sigma\cong \sL_m^{-1}$. Then the points $h_*\sL_m$ are dense in $N(E)$. Thus the Zariski closure $Z$ of the $\sL_m$ in $\Pic^\tau(Y)$ has the property that $h_*Z=N(E)$ by the argument we had before. Thus there is a $L\in Z$ such that $h_*L=E$. This finishes the characterization of the second case. \\[.1cm]
To show that one has either case i) or ii), one has to appeal to the classification of finite groups of $SL(2, \bar \F_p)$: one deduces from  Dickson's Theorem \cite[Hauptsatz~8.27]{Hup} that there are only 2 infinite families of finite subgroups of $SL(2, \bar \F_p)$ such that the natural rank 2 representation is geometrically irreducible, the dihedral subgroups $D_m$ and the finite subgroups $SL(2,  \F_{q_i})$. Arguing again that fixing a finite group $H$, there are finitely many surjections $\pi^{{\rm \et}}(X,x)\surj H$, one shows the dichotomy. 

\end{proof}

\section{Some remarks and questions}
\subsection{}\label{subs4.1} Using the effective Lefschetz properties \cite[Corollaire~3.4]{SGA2}, the property that stratified bundles uniquely lift to the formal neighbourhood of a smooth divisor \cite[Proposition~1.5]{G}, and the fact that coherent $\sD_X$-modules which are  $\sO_X$-coherent are locally free (\cite[Lemma~6]{dS}),
one shows that if $Y\subset X$ is a smooth ample divisor on a smooth projective variety defined over a perfect field of characteritic $p>0$, and if $y\in Y(k)$, then the homomorphism $\pi^{{\rm strat}}_1(Y, y)\to \pi_1^{{\rm strat}}(X,x)$ induced by the restriction of bundles is an isomorphism, if ${\rm dim}_k (Y)\ge 2$.  But in order to reduce the proof of the main theorem of this article to surfaces, one would need the full strength of Langer's theorem on corepresentatibility, which is hard to extract from Gieseker's article. So we did not present the argument via this reduction. 
\subsection{} \label{subs4.2} Simpson constructed in \cite{Simp} quasi-projective moduli schemes of semi-simple bundles with flat connection on a complex smooth projective variety. It is not unlikely that his ideas, combined with Langer's methods in characteristic $p>0$ (\cite{L}) and Proposition \ref{prop2.3},  yield the existence of quasi-projective moduli schemes of stratified bundles. Those moduli, aside of their own interest,  could then be used directly to prove Theorem \ref{thm1.1}.
\subsection{} \label{subs4.3} As mentioned in the introduction, Grothendieck's theorem over $\C$ has nothing to do with stability questions, and applies to the algebraic completion of the topological fundamental group of any smooth complex quasi-projective manifold. 
If we translate his theorem via the Riemann-Hilbert correspondence on $X$ smooth complex but not necessarily proper, then he shows that if the profinite completion  topological fundamental group is trivial, there is no nontrivial regular singular bundles with flat connection. 
The analog in characteristic $p>0$ over  a quasi-projective smooth variety $X$ would
relate  Grothendieck's \'etale fundamental group over $\bar k$ and tame stratified bundles. 
This would require a completely different proof. At any rate, we do not even have at disposal   a  theory of good lattices in $j_*E$, where $E$ is stratified on $X$ and $j:X \to \hat{X}$ is a normal  compactification. 
Unfortunately,  we can not say anything on this subject. 
\subsection{} \label{subs4.4} The proof of the main theorem may be seen as an application of Hrushovski's theorem. On one hand, it is very nice to see how his profound theorem works concretely for some natural question coming from algebraic geometry. On the other hand, to have another proof anchored in algebraic geometry would perhaps shed more light on the original problem. 
\bibliographystyle{plain}
\renewcommand\refname{References}

\end{document}